\documentclass[11pt,leqno]{amsart}
\usepackage{amsmath,amssymb,amsthm}
\newtheorem{theorem}{Theorem}[section]
\newtheorem{lemma}[theorem]{Lemma}

\newtheorem{proposition}[theorem]{Proposition}
\newtheorem{corollary}[theorem]{Corollary}
\theoremstyle{definition}
\newtheorem{definition}[theorem]{Definition}
\newtheorem{example}[theorem]{Example}

\newtheorem{remark}[theorem]{Remark}
\newtheorem*{Acknowledgement}{\textnormal{\textbf{Acknowledgement}}}
\newtheorem*{Conflict of interest statement and data availability}{\textnormal{\textbf{Conflict of interest statement and data availability}}}
\numberwithin{equation}{section}
\def\fnote#1{\footnote}

\def\ignora#1{}
\def\n3#1{\left\vert  \! \left\vert \! \left\vert \, #1 \, \right\vert \!
  \right\vert \! \right\vert }

\begin{document}

\title{ Non-rough norms and dentability in spaces of operators }
%\author{}
%\address{Universidad de Granada, Facultad de Ciencias.
%Departamento de Matem\'{a}tica Aplicada, 18071-Granada (Spain)}

\author{Susmita Seal$ ^{1}$, Sudeshna Basu$^{2}$, Julio Becerra Guerrero$^{3}$ and Juan Miguel Villegas Yeguas$^{4}$}

\address{{$^{1}$} Susmita Seal, 
	Department of Mathematics, 
	Ramakrishna Mission Vivekananda Education and Research Institute, 
	Belur Math,  Howrah 711202,
	West Bengal, India}

\email{susmitaseal1996@gmail.com}
	
\address{{$^{2}$}   Sudeshna Basu,
	Department of Mathematics and Statistics, 
				Loyola University, 
				Baltimore, MD 21210 and
				Department of Mathematics,
				George Washington University,
				Washington DC 20052 USA  }

\email{sudeshnamelody@gmail.com} 

\address{{$^{3}$} Julio Becerra Guerrero, 
		Universidad de Granada,
		Facultad de Ciencias.
		Departamento de An\'{a}lisis Matem\'{a}tico, 18071-Granada,
		Spain} 
	
	\email{juliobg@ugr.es}

\address{{$^{4}$} Juan Miguel Villegas Yeguas,    
	Universidad de Granada, 
	Facultad de Ciencias.
	Departamento de An\'{a}lisis Matem\'{a}tico, 18071-Granada,
	Spain} 

\email{ juanmivy@correo.ugr.es}

\subjclass{46B20, 46B28}
\keywords{Slices, Stability, Tensor product, Space of operators, Huskable, Denting, Dentable, Small Combination of Slices.}
\date{}
\sloppy

\maketitle \markboth{J. Becerra, S. Basu, S. Seal and J. M. Villegas
}{ Non-rough norms in   operator spaces }

\bigskip

\begin{abstract} 
In this work, we study  non-rough norms in $L(X,Y),$ the space of bounded linear operators between Banach spaces $X$ and $Y.$ We prove that $L(X,Y)$ has
	non-rough norm if and only if $X^*$ and $Y$ have non-rough norm. We show that the injective tensor product $X\hat{\otimes}_{\varepsilon} Y$ has non-rough norm if and only if both $X$ and $Y$ have non-rough norm. We also give an example to show that non-rough norms are not stable under projective tensor product. We also study a related concept namely the small diameter properties in the context of $L(X,Y)^*.$ These results  leads to a discussion on 
	stability of the small diameter properties for projective and injective tensor product spaces.
	
\end{abstract}

\section{Introduction.}

We will consider only real Banach spaces.
Throughout the paper, we will denote the closed unit ball by $B_X$, the unit sphere by $S_X$ and the closed ball of radius $r >0$ and center $x$ by $B_X(x, r).$ 
 $X^*$ stand for the topological dual of $X$. We refer the reader to the monograph \cite{Bo1} for notions of convexity theory that we will be using here. 

\begin{definition}
For a Banach space $X,$
\begin{enumerate}
\item a subset $C \subset B_{X^*}$ is a norming set for $X,$ if $\|x\| =
\sup\{|x^*(x)| : x^* \in C\}$ for all $x \in X$. A closed subspace $F
\subset X^*$ is a norming subspace, if $B_F$ is a norming set for $X$.
\item a slice in the bounded set $C\subset X$ is the set $$S(C, x^*, \alpha) = \{x \in C : x^*(x) > \mbox{sup}~ x^*(C) - \alpha \},$$ where $x^*\in S_{X^*}$ and $0<\alpha <1$. Analogously we can define $w^*$-slices in $X^*$ by choosing the determining functional from the predual space.
\item a function $f:X\rightarrow \mathbb{R}$ is called G$\hat{a}$teaux differentiable at $x\in X,$ if for each $h\in X,$ $$f'(x)(h)=\lim\limits_{t\rightarrow 0} \frac{f(x+th)-f(x)}{t}$$ exists and is a linear continuous function in h. The functional $f'(x)$ is called the G$\hat{a}$teaux derivative of f at x.
\item a function $f:X\rightarrow \mathbb{R}$ is called Fr$\acute{e}$chet differentiable at $x\in X,$ if there exists $f'(x)\in X^*$ such that 
$$\lim\limits_{y\rightarrow 0} \frac{f(x+y)-f(x)-f'(x)(y)}{\|y\|}=0.$$
The functional $f'(x)$ is called the Fr$\acute{e}$chet derivative of f at x.
\end{enumerate}
\end{definition} 

A Fr$\acute{e}$chet differentiable map is always G$\hat{a}$teaux differentiable but not the converse \cite{Ph}.
\begin{definition} 
  The
roughness of $X$ at $u\in S_X$ is 
$$\eta(X,u)=\limsup\limits_{\Vert h\Vert\rightarrow 0} \frac{\Vert u+h\Vert+\Vert u-h\Vert-2}{\Vert h\Vert}.$$
	For $\varepsilon>0$, norm of $X$ is said to be
	$\varepsilon$-rough, if $\eta(X,u)\geq \varepsilon$ for every $u\in S_X$. We say that norm of $X$ is rough,
	 if it is $\varepsilon$-rough for some $\varepsilon >0$ and it is non-rough otherwise.
	 \end{definition}
Rough norms were first studied by Day \cite{Da}, Kurzweil \cite{Ku}, Leach and Whitfield \cite{LW}, John and Zizler \cite{JZ} and G. Godefroy in \cite{Go}.
Non-rough norms have been used in order to characterize when a Banach space is Asplund. In fact, G. Godefroy
shows in \cite{Go}, that a Banach space $X$ is Asplund  if and only if every equivalently norm in $X$ is non-rough. It was proved in \cite{DGZ}, for  $u\in S_X$,  $\eta(X,u)=0$ if and only if the norm is Fr\'echet differentiable at $u$.
There is a Banach space X whose norm is both rough and smooth. There is also a  Banach space having a non rough norm with no
point of Gateaux differentiability, for details, see  \cite{JZ}.

To prove our aim 
let us recall the following three properties already studied in \cite{Ba}, \cite{BR} and \cite{BS}. 
 
\begin{definition}
A Banach space $X$ has 
%\begin{itemize}
\begin{enumerate}
	
	\item  Ball Dentable Property ($BDP$) if $B_X$ has  
	slices of arbitrarily small diameter. 
	
	\item  Ball Huskable Property ($BHP$) if $B_X$ 
	has nonempty relatively weakly open subsets of arbitrarily small diameter.
	
	\item  Ball Small Combination of Slice Property ($BSCSP$) 
	if  $B_X$ has convex combination of slices of arbitrarily 
	small diameter.
\end{enumerate}
%	\end{itemize}
\end{definition}
One can analogously define $w^*$-$BDP$, $w^{*}$-$BHP$ and $w^{*}$-$BSCSP$ in a dual space. The implications among all these properties are clearly described in the following diagram.
%See \cite{Ba}, \cite{BR} and \cite{BS} for more details.

\newpage
 
$$ BDP \Longrightarrow \quad BHP \Longrightarrow \quad  BSCSP$$ $ \quad\quad\quad\quad\quad\quad\quad\quad\quad\quad \Big \Uparrow \quad \quad\quad\quad\quad \Big \Uparrow \quad \quad\quad\quad\quad \Big \Uparrow$  $$ w^*-BDP \Longrightarrow  w^*-BHP \Longrightarrow  w^*-BSCSP$$

but reverse implications are not true in general \cite{BS}. 

It is clear that Radon-Nikod$\acute{y}$m Property ($RNP$) implies $BDP$, Point of Continuity Property ($PCP$) implies $BHP$ and Strongly Regular ($SR$) implies $BSCSP$. Therefore, the properties $BDP$, $BHP$ and $BSCSP$ are "localised" (to the closed unit ball) versions of the three geometric properties $RNP$, $PCP$ and $SR$. For more details on $RNP$, $PCP$ and $SR$, see \cite{GGMS}.

The space of all bounded bilinear forms defined on $X\times Y$ is denoted by $B(X\times Y).$
Let $x\in X$,  $y\in Y$, then  $x\otimes y: B(X\times Y)\rightarrow \mathbb{R} $ is defined as $A \mapsto A(x,y)$, where $A\in B(X\times Y).$
%$<A,x\otimes y>= A(x,y)$  $A\in B(X\times Y)'.$ 
Also we use $L(X, Y)$ to denote the space of all bounded and linear operators from $X$ into $Y$.
The tensor product of $X$ and $Y$ is denoted by $X\otimes Y$ and it is the subspace spanned by all elements $x\otimes y$ where $x\in X$ and $y\in Y.$

\begin{definition}
The projective norm $\|.\|_{\pi}$ on tensor product $X\otimes Y$ is given by  $\|u\|_{\pi}= \inf \{ \sum_{i=1}^{n} \|x_i\|  \|y_i\| : u=\sum_{i=1}^{n} x_i \otimes y_i\}.  $
\end{definition}

 The pair $(X\otimes Y, \|.\|_{\pi})$ is denoted by $X \otimes_{\pi} Y$ and its completion is called projective tensor product, denoted by $X \hat{\otimes}_{\pi} Y.$ The space $ L(X,Y^*)$ is linearly isometric to the
topological dual of $X\widehat{\otimes}_\pi Y$.

\begin{definition}
 The injective norm $\|.\|_{\varepsilon}$ on tensor product $X\otimes Y$ is given by  $\|u\|_{\varepsilon}= \sup \{ |\sum_{i=1}^{n} x^*(x_i) y^*(y_i)| : u=\sum_{i=1}^{n} x_i \otimes y_i, x^*\in S_{X^*},y^*\in S_{Y^*}\}.  $
 \end{definition}
 The pair $(X\otimes Y, \|.\|_{\varepsilon})$ is denoted by $X \otimes_{\varepsilon} Y$ and its completion is called injective tensor product, denoted by $X \hat{\otimes}_{\varepsilon} Y.$ Every $u\in X \hat{\otimes}_{\varepsilon} Y$ can be considered as an element $T_u\in L(X^*,Y)$ which is $w^*-w$ continuous. In particular, $y^* \circ T_u\in X$ for all $y^*\in Y^*.$

%Let us observe that having the Radon-Nikodym property  implies BDP, because this property is
%characterized in terms of the existence of slices with diameter
%arbitrarily small. Also it is known that the projective tensor
%product of Banach spaces with the RNP is not necessarily a RNP
%space \cite{BP}.   It is proved in \cite{ABGRP} that  SD2P (every convex combination of
%slices in the unit ball of a Banach space has diameter $2$) is stable for projective tensor products \cite{ALN}. 
%Then it is natural open question if the BDP is stable for the projective tensor products.

 In this paper we study the existence of non-rough 
norms in spaces of operators. We prove that the space of operators $L(X,Y)$ has non-rough norm if and only if $X^*$ and $Y$ have non-rough norm. Also we get some necessary conditions so that  $L(X,Y)^*$ has $w^*$-$BSCSP.$  We also show that the injective tensor product $X\hat{\otimes}_{\varepsilon} Y$ has non rough norm if and only if both $X$ and $Y$ have non-rough norm. We give counter examples to show non-rough norm is not stable under projective tensor product.
  In some of our proofs, we use similar techniques as in \cite{BGLPRZ2}, \cite{LLZ1}. 
  
\section{Small diameter property in projective tensor products.}

Let us recall the following results to prove the stability results of $BDP$ in projective tensor product.
\newpage

 \begin{lemma}\label{rough} \cite{DGZ} 
 	Let $X$ be Banach space. The following are equivalent.
 	\begin{enumerate}
 		\item $X$ is non-rough.
 		\item $X^*$ has $w^*$-$BDP.$
 	\end{enumerate} 
 	
 	 \end{lemma}
 
 \begin{theorem} \label{bs 1st}
\cite{BS} A Banach space $X$ has $BDP$ (resp. $BHP$ , $BSCSP$) if and only if $X^{**}$ has $w^*$-$BDP$ (resp. $w^*$-$BHP$, $w^*$-$BSCSP$). 
\end{theorem}

\begin{lemma}\label{densidad-principal}
	Let $X,Y$ be Banach spaces, and  $\Gamma \subset S_{X}$ with $B_X=\overline{co}(\Gamma )$ and $\Omega \subset S_{Y^*}$ a norming subset. We consider  $\Gamma '\subset S_{X^*}$ and $\Omega '\subset S_Y$ such that $\Vert x\Vert =\sup \{x^*(x):x^*\in \Gamma ' \}$ and  $\Vert y^*\Vert =\sup \{y^*(y):y\in \Omega ' \}$ for every $x\in \Gamma$ and $y^* \in \Omega$. If $H$ is a closed subspace of
	$L(X,Y)$ such that $\Gamma '\otimes \Omega '\subseteq H$, then 
	$B_{H^*}=\overline{co}^{w^*}(\Gamma \otimes \Omega )$.
\end{lemma}
\begin{proof} Since $B_X=\overline{co}(\Gamma )$ and $\Omega \subset S_{Y^*}$ is a norming subset, it is clear that
	$$ \Vert T\Vert =\sup \{y^*(T(x)):x\in \Gamma , \ y^*\in \Omega \},$$ for every $T\in H$. Thus, $\Gamma\otimes \Omega $ is a norming set for $H,$ where $(x\otimes y^*)(T):=y^*(T(x))$ and $\Vert x\otimes y^*\Vert =1$ since $\Gamma '\otimes \Omega '\subseteq H$, for $x\in \Gamma $ and $ y^*\in \Omega$. Hence, by Hahn-Banach separation theorem, $B_{H^*}=\overline{co}^{w^*}(S_{X^*}\otimes S_{Y^*}).$
\end{proof}

\begin{corollary} \label{densidad} Let $X,Y$ be Banach spaces. We consider $H$ a closed subspace of $L(X,Y)$ such that $X^*\otimes Y\subseteq H$. Then we have that $B_{H^*}=\overline{co}^{w^*}(S_X\otimes S_{Y^*})$.
\end{corollary}

\begin{corollary}\label{inj lemma}
Let $X,Y$ be Banach spaces and $H$ be a closed subspace of $L(X^*,Y)$ such that $X\otimes Y\subset H.$ Then $B_{H^*}=\overline{co}^{w^*}(S_{X^*}\otimes S_{Y^*}).$
\end{corollary}

The next result gives us  
stability of BDP in spaces of operators.

\begin{theorem}\label{teostrong2}
Let $X,Y$ be Banach spaces. Let $H$ be a closed subspace of $L(X,Y)$ such
that $X^*\otimes Y\subseteq H$. Then  $H^*$ has $w^*$-$BDP$ if and only if $X^{**}$ and $Y^*$ have $w^*$-$BDP.$
\end{theorem}
\begin{proof}
Let us assume that $H^*$ has the $w^*$-$BDP.$ Let $\varepsilon >0$. Then there exist $T\in S_H$ and $\alpha >0$ such that  $diam (S(B_{H^*},T,\alpha ))<\varepsilon .$ By Corollary \ref{densidad}, there exist $x_0\in S_X, \ y^*_0\in S_{Y^*}$ such that $y^*_0(T(x_0))>1-\alpha$. We define the slices $S_1:=\{x\in B_X:y^*_0(T(x))>1-\alpha\}$ and $S_2:=\{y^*\in B_{Y^*}:y^*(T(x_0))>1-\alpha\}$. It is clear that $S_1\otimes y^*_0, x_0\otimes S_2\subset S(B_{H^*},T,\alpha )$, and hence $diam (S_1)<\varepsilon$ and $diam (S_2)<\varepsilon$. Then $X^{**}$ and $Y^*$ have $w^*$-$BDP.$ 
Now, we assume that  $X^{**}$ and $Y^*$ have $w^*$-$BDP.$ Let $\varepsilon >0$. Then there exist $f\in S_{X^*}$, $y\in S_Y$ and $\alpha >0$ such that  $$diam (S(B_X, f, \alpha))< \varepsilon \  \text{and} \  diam (S(B_{Y^*}, y, \alpha ))< \varepsilon . $$ We prove that $diam (S(B_{H^*}, f \otimes y, \delta^2 ))<  8 \varepsilon, $ where $0<\delta<\min\{\alpha, \varepsilon\}.$ Given  $z \in S(B_{H^*}, f \otimes y, \delta^2)$, we can assume that $$z= \sum_{j=1}^n \lambda_j x_j \otimes y_j^*$$ where $x_j \in B_X$, $y_j^* \in B_{Y^*}$ $\lambda_j \geq 0$ and $\sum_{j=1}^n \lambda_j=1$. We fix elements $x_0\in S(B_X, f, \alpha)$ and $y_0^*\in S(B_{Y^*}, y, \alpha )$.
We denote by
$$I= \{ j : 1 \leq j \leq n  , f(x_j)y(y_j^*)< 1-\delta  \} \qquad I'= \{1,...,n\} \backslash I$$

Then we have

\begin{equation} \notag
\begin{split}
1 - \delta^2  \leq (f \otimes y) (z) = 
   \sum_{j \in I} \lambda_j f(x_j) y(y_j^*) + \sum_{j \in I'} \lambda_j f(x_j) y(y_j^*)  \\
\leq  (1-\delta) \sum_{j \in I} \lambda_j + \sum_{j \in I'} \lambda_j = (1-\delta) \sum_{j \in I} \lambda_j + 1 - \sum_{j \in I} \lambda_j = 1 - \delta \sum_{j \in I} \lambda_j
\end{split}
\end{equation}

and so $\sum_{j \in I} \lambda_j < \delta$ and $  1 - \delta < \sum_{j \in I'} \lambda_j .$ We conclude that

\begin{equation} \notag
\begin{split}
\Vert z - \sum_{j \in I'} \lambda_j x_j \otimes y_j^*\Vert  = \Vert \sum_{j \in I \cup I' } \lambda_j x_j \otimes y_j^* - \sum_{j \in I'} \lambda_j x_j \otimes y_j^*\Vert  \\   = \Vert \sum_{j \in I } \lambda_j x_j \otimes y_j^*\Vert  \leq \sum_{j \in I } \lambda_j \Vert x_j\Vert \Vert y_j^*\Vert  \leq \sum_{j \in I } \lambda_j ,
\end{split}
\end{equation}

and hence

\begin{equation} \label{eq4}
\Vert z - \sum_{j \in I'} \lambda_j x_j \otimes y_j\Vert < \delta
\end{equation}

Since  $f(x_j)y(y_j^*) \geq 1- \delta $ for all $j \in I'$, we have that $f(x_j) \geq 1-\delta $ and $y(y_j^*) \geq 1-\delta $ for all $j \in I'$. Hence
 $\Vert x_0-x_j\Vert <\varepsilon$ and $\Vert y_0^*-y_j^*\Vert <\varepsilon$ for all $j \in I'$.

As a consequence

\begin{equation} \notag
\begin{split}
\Vert z-x_0\otimes y_0^*\Vert = \Vert \sum_{j=1}^n \lambda_j (x_j \otimes y_j^*-x_0\otimes y_0^*)\Vert      \leq 2 \delta + \Vert \sum_{j \in I' } \lambda_j (x_j \otimes y_j^*-x_0\otimes y_0^*)\Vert  \\  \leq 2 \delta + \Vert \sum_{j \in I' } \lambda_j (x_j \otimes y_j^*+x_0 \otimes y_j^*-x_0 \otimes y_j^*-x_0\otimes y_0^*)\Vert  \\  \leq 2 \delta +  \sum_{j \in I' } \lambda_j (\Vert (x_j -x_0)\otimes y_j^*\Vert + \Vert x_0 \otimes (y_j^*- y_0^*)\Vert ) \leq 2\delta +  \sum_{j \in I' } \lambda_j 2\varepsilon \leq 4\varepsilon
\end{split}
\end{equation}

We conclude that $diam (S(B_{H^*}, f \otimes y, \delta^2 ))< 8\varepsilon $.
\end{proof}
%Let $X$ be a Banach space. For $u\in S_X$ we define the roughness of $X$ at $u$, denoted by $\eta(X,u)$, by the equality $$\eta(X,u)=\limsup\limits_{\Vert h\Vert\rightarrow 0} \frac{\Vert u+h\Vert+\Vert u-h\Vert-2}{\Vert h\Vert}$$

\begin{corollary}\label{BDP}
Let $X,Y$ be Banach spaces. Let $H$ be a closed subspace of $L(X,Y)$ such
that $X^*\otimes Y\subseteq H$. Then norm of $H$ is non-rough if and only if norm of $X^*$ and $Y$ are non-rough.
\end{corollary}

Since the dual of the projective tensor product of Banach spaces $X$ and $Y$ is the space of operators $L(X,Y^*)$ we have the following :

\begin{corollary}\label{SD2P}
The projective tensor product,  $X\widehat{\otimes}_\pi Y$, of two
Banach spaces $X$ and $Y$, has BDP if and only if $X$ and $Y$
have BDP.
\end{corollary}

\begin{proof}
Taking into account that the dual of the projective tensor product $X\widehat{\otimes}_\pi Y$ is $L(X,Y^*)$ and Theorem $\ref{teostrong2},$ we can conclude $(X\widehat{\otimes}_\pi Y)^{**}$ has $w^*$-$BDP$ if and only if $X^{**}$ and $Y^{**}$ have $w^*$-$BDP.$ Applying Theorem $\ref{bs 1st}$ we get our desired result.
\end{proof}

\begin{remark}
The following example shows that the assumptions cannot be weakened, in Corollary \ref{BDP}. Indeed, Let  $Z:=C(K)\widehat{\otimes}_\pi \ell_2$,
for  any infinite compact Hausdorff topological space $K$. By \cite[Theorem 4.1]{ABGRP}, every slice of the unit ball of $Z$ has diameter two even though $\ell_2$ has BDP.
\end{remark}

Another consequence of the Lemma \ref{rough} and the proof of Theorem \ref{teostrong2}  is the following version of the result \cite[Theorem 1.8]{RS}.

\begin{corollary}\label{Frechet}
	Let $X,Y$ be Banach spaces. Let  $H$ be a closed subspace of $L(X,Y)$ such
	that $X^*\otimes Y\subseteq H$. Let $x_0^*\otimes y_0\in S_{X^*}\otimes S_Y$. Then  the norm of $H$ is Fr\'echet differentiable at $x_0^*\otimes y_0$  if and only if the norms of $X^*$ and $Y$ are Fr\'echet differentiable at $x_0^*$ and $y_0$, respectively.
\end{corollary}

Our next result is a necessary condition for the dual of a space of
operators to have $w^*$-$BSCSP.$

\begin{proposition}\label{necesaria} Let $X$ and $Y$ be Banach spaces and $H$ be a closed subspace of $L(X,Y)$ such that $X^* \otimes Y\subset H.$ Also assume that $X$ has $BDP$ and $Y^*$ has $w^*$-$BSCSP,$ then $H^*$ has $w^*$-$BSCSP.$
\end{proposition}
\begin{proof}
Let $\varepsilon>0.$ Since $X$ has $BDP$ and $Y^*$ has $w^*$-$BSCSP$, then let\\ $S(B_X, x^*,\alpha)$ be a slice in $B_X$ with diameter less than $\varepsilon$ and\\ $\sum_{i=1}^{n} \lambda_i S(B_{Y^*},y_i,\delta)$ be a convex combination of slices in $B_{Y^*}$ with diameter less than $\varepsilon.$ Choose $\beta$ such that $0< 2\beta <\min\{\delta,\alpha,\varepsilon\}.$ Also choose $x_0\in S(B_X, x^*,\alpha) \bigcap S_X.$ Consider $\sum_{i=1}^{n} \lambda_i S(B_{H^*},x^*\otimes y_i,\beta^2),$ a convex combination of $w^*$ slices in $B_{H^*}.$ 
Let $ \sum_{i=1}^{n} \lambda_i f_i', \sum_{i=1}^{n} \lambda_i g_i' \in \sum_{i=1}^{n} \lambda_i S(B_{H^*},x^*\otimes y_i,\beta^2).$ Since $H$ is closed subspace of $L(X,Y)$ with $X^* \otimes Y\subset H$ then 
%$co(S_X\otimes S_{Y^*})$ is norming set for $H$ 
 $B_{H^*}=\overline{co}^{w^*} (S_X\otimes S_{Y^*})$ (by 
 %\cite{BGLPRZ2}, Lemma 2.3).
 Lemma $\ref{densidad} $) 
Choose $h\in S_H$ such that $$\Big | \sum_{i=1}^{n} \lambda_i f_i'(h)-\sum_{i=1}^{n} \lambda_i g_i'(h)\Big |>\|\sum_{i=1}^{n} \lambda_i f_i'-\sum_{i=1}^{n} \lambda_i g_i'\|-\varepsilon.$$ Since $co(S_X\otimes S_{Y^*})$ is $w^*$dense in $B_{H^*},$  for each $i=1,2,\ldots,n$  
$$co(S_X\otimes S_{Y^*}) \cap \Big (W(f_i',h,\varepsilon)\cap S(B_{H^*},x^*\otimes y_i,\beta^2)\Big ) \neq \emptyset $$ 
$$co(S_X\otimes S_{Y^*}) \cap \Big (W(g_i',h,\varepsilon)\cap S(B_{H^*},x^*\otimes y_i,\beta^2)\Big ) \neq \emptyset $$  
 % Thus without loss of generality let us assume that \\
For all $i=1,2,\ldots,n$ we choose
 $$f_i= \sum_{k=1}^{m_i} \gamma_{(k,i)} x_{(k,i)} \otimes y^*_{(k,i)} \in co(S_X\otimes S_{Y^*}) \cap \Big (W(f_i',h,\varepsilon)\cap S(B_{H^*},x^*\otimes y_i,\beta^2)\Big ) $$  where $0<\gamma_{(k,i)}\leqslant 1$ with $\sum_{k=1}^{m_i} \gamma_{(k,i)}=1$ and  $x_{(k,i)} \in S_X, y^*_{(k,i)}\in S_{Y^*}$  $\forall k=1,2,\ldots,m_i$
 $$g_i= \sum_{k=1}^{n_i} \gamma'_{(k,i)} u_{(k,i)} \otimes v^*_{(k,i)}  \in co(S_X\otimes S_{Y^*}) \cap \Big (W(g_i',h,\varepsilon)\cap S(B_{H^*},x^*\otimes y_i,\beta^2)\Big )$$  where  $0<\gamma'_{(k,i)}\leqslant 1$ with $\sum_{k=1}^{n_i} \gamma'_{(k,i)}=1$ and  $ u_{(k,i)} \in S_X, v^*_{(k,i)} \in S_{Y^*}$  $\forall k=1,2,\ldots,n_i.$ 

For all $i=1,2,\ldots,n,$ define
$$P_i=\{(k,i)\in \{1,2,\ldots,m_i\} \times \{i\} : (x^*\otimes y_i)\Big (x_{(k,i)}\otimes y^*_{(k,i)}\Big )>1-\beta\}$$
$$Q_i=\{(k,i)\in \{1,2,\ldots,n_i\} \times \{i\} : (x^*\otimes y_i)\Big (u_{(k,i)}\otimes v^*_{(k,i)}\Big )>1-\beta\}$$
Then,
$$1-\beta^2 < (x^*\otimes y_i) \Big (\sum_{k=1}^{m_i} \gamma_{(k,i)} \Big ( x_{(k,i)} \otimes y^*_{(k,i)}\Big ) \Big ) \hspace{5 cm}$$
$$\hspace{2 cm}=\sum_{(k,i)\in P_i} \gamma_{(k,i)} (x^*\otimes y_i)  \Big (x_{(k,i)} \otimes y^*_{(k,i)}\Big )  + 
\sum_{(k,i)\notin P_i} \gamma_{(k,i)} (x^*\otimes y_i) \Big (x_{(k,i)} \otimes y^*_{(k,i)}\Big ) $$
$$\leqslant \sum_{(k,i)\in P_i} \gamma_{(k,i)} + (1-\beta) \sum_{(k,i)\notin P_i} \gamma_{(k,i)} \hspace{4 cm}$$
$$=1-\sum_{(k,i)\notin P_i} \gamma_{(k,i)} +(1-\beta)\sum_{(k,i)\notin P_i} \gamma_{(k,i)} \hspace{4 cm}$$
i.e. $\sum_{(k,i)\notin P_i} \gamma_{(k,i)} <\beta$

Also observe, 
$$ y_i \Big ( \sum_{(k,i)\in P_i} \gamma_{(k,i)} y^*_{(k,i)} \Big )
\geqslant (x^*\otimes y_i) \Big (\sum_{(k,i)\in P_i} \gamma_{(k,i)} \Big ( x_{(k,i)} \otimes y^*_{(k,i)}\Big ) \Big )\hspace{3 cm}$$
$$=\sum_{(k,i)\in P_i} \gamma_{(k,i)} (x^*\otimes y_i) \Big ( x_{(k,i)} \otimes y^*_{(k,i)}\Big )$$
$$>(1-\beta)^2\hspace{4 cm}$$
$$>1-\delta \hspace{4 cm}$$
Therefore, $\phi_i := \sum_{(k,i)\in P_i} \gamma_{(k,i)} y^*_{(k,i)} \in S(B_{Y^*},y_i,\delta)$

Again for $(k,i)\in P_i,$
$$ x^*\Big ( x_{(k,i)} \Big ) 
\geqslant (x^*\otimes y_i) \Big ( x_{(k,i)} \otimes y^*_{(k,i)}\Big )
>1-\beta >1-\alpha$$
Therefore, $x_{(k,i)}\in S(B_X, x^*,\alpha)$ 
%and thus it follows that $\| x_{(k,i)} - x_0\|<\varepsilon$

Similarly, $\psi_i := \sum_{(k,i)\in Q_i} \gamma'_{(k,i)} v^*_{(k,i)} \in S(B_{Y^*},y_i,\delta)$ and $u_{(k,i)}\in S(B_X, x^*,\alpha)$

Hence we have $\| x_{(k,i)} - x_0\|<\varepsilon$ and $\| u_{(k,i)} - x_0\|<\varepsilon$

Now, for all $i=1,2,\ldots,n$
$$\|f_i - x_0\otimes \phi_i\|
\leqslant \|f_i - \sum_{(k,i)\in P_i} \gamma_{(k,i)} \Big ( x_{(k,i)} \otimes y^*_{(k,i)}\Big )\| + \|\sum_{(k,i)\in P_i} \gamma_{(k,i)} \Big ( x_{(k,i)} \otimes y^*_{(k,i)}\Big ) - x_0\otimes \phi_i\|$$
$$\hspace{1 cm}=\|\sum_{(k,i)\notin P_i} \gamma_{(k,i)} \Big ( x_{(k,i)} \otimes y^*_{(k,i)}\Big )\| + \|\sum_{(k,i)\in P_i} \gamma_{(k,i)} \Big ( x_{(k,i)}- x_0\Big ) \otimes y^*_{(k,i)}\|$$
$$\leqslant \sum_{(k,i)\notin P_i} \gamma_{(k,i)} + \sum_{(k,i)\in P_i} \gamma_{(k,i)}  \|x_{(k,i)}- x_0\| \|y^*_{(k,i)}\| \hspace{2 cm}$$
$$<\beta + \varepsilon \hspace{8 cm}$$
Similarly, $\|g_i - x_0\otimes \psi_i\| < \beta + \varepsilon$

Hence, $$\|\sum_{i=1}^{n} \lambda_i f_i - \sum_{i=1}^{n} \lambda_i g_i\|
\leqslant \| \sum_{i=1}^{n} \lambda_i (f_i- x_0\otimes \phi_i)\| + \|\sum_{i=1}^{n} \lambda_i x_0 \otimes (\phi_i-\psi_i)\| + \|\sum_{i=1}^{n} \lambda_i (g_i- x_0\otimes \psi_i)\|$$
$$< \sum_{i=1}^{n} \lambda_i (\beta +\varepsilon) + \varepsilon + \sum_{i=1}^{n} \lambda_i (\beta + \varepsilon) \hspace{3 cm}$$
$$= 2\beta + 3\varepsilon 
\leqslant 4\varepsilon \hspace{6 cm}$$
Finally, $\|\sum_{i=1}^{n} \lambda_i f_i' - \sum_{i=1}^{n} \lambda_i g_i'\|$
$$< \Big | \sum_{i=1}^{n} \lambda_i f_i'(h)-\sum_{i=1}^{n} \lambda_i g_i'(h)\Big |+\varepsilon \hspace{9 cm} $$
$$\leqslant \Big | \sum_{i=1}^{n} \lambda_i f_i'(h)- \sum_{i=1}^{n} \lambda_i f_i (h) \Big | + \Big | \sum_{i=1}^{n} \lambda_i f_i(h)- \sum_{i=1}^{n} \lambda_i g_i (h) \Big | + \Big | \sum_{i=1}^{n} \lambda_i g_i(h)- \sum_{i=1}^{n} \lambda_i g'_i (h) \Big | +\varepsilon$$
$$<\varepsilon + 4\varepsilon + \varepsilon + \varepsilon
=7\varepsilon \hspace{10.5 cm}$$

Therefore, $diam\big (\sum_{i=1}^{n} \lambda_i S(B_{H^*},x^*\otimes y_i,\beta^2)\Big )\leqslant 7\varepsilon$ and so $H^*$ has $w^*$-$BSCSP.$ 
\end{proof}

Since  $X$ and $Y$are symmetric  in the proof of the
above result, we have the following:

\begin{corollary}\label{simetri} Let $X,Y$ be Banach spaces such that norm of $Y$ is non-rough and $X$ has the $BSCSP.$ Let  $H$ be a closed subspace of $L(X,Y)$ such that $X^*\otimes Y\subseteq H$. Then  $H^*$ has the $w^*$-$BSCSP.$
\end{corollary}

\begin{corollary}\label{bscsp prop proj}
 Let $X,Y$ be Banach spaces. Assume that $X$ has $BDP$ and $Y$ has $BSCSP.$ Then $X \hat{\otimes}_{\pi} Y$ has $BSCSP.$
 \end{corollary}
 
 \begin{proof}
Taking into account that the dual of the projective tensor product $X\widehat{\otimes}_\pi Y$ is $L(X,Y^*)$ and Theorem $\ref{necesaria},$ we can conclude that if $X$ has $BDP$ and $Y^{**}$ has $w^*$-$BSCSP,$ then  $(X\widehat{\otimes}_\pi Y)^{**}$ has $w^*$-$BSCSP.$ Applying Theorem $\ref{bs 1st}$ we get our desired result.
\end{proof}

Small changes in the proof of the Proposition \ref{necesaria}, allow to obtain the following result

\begin{proposition} \label{bscsp in inj tensor}
	Let $X$ and $Y$ be Banach spaces and $H$ be a closed subspace of $L(X^*,Y)$ such that $X \otimes Y\subset H.$ Also assume that $X^*$ has $w^*$-$BDP$ and $Y^*$ has $w^*$-$BSCSP,$ then $H^*$ has $w^*$-$BSCSP.$
\end{proposition}

%\begin{proof}
%Consider $H= (X \hat{\otimes}_{\pi} Y)^* \cong L(X,Y^*).$  Since $X^*\otimes Y^* \subset L(X,Y^*),$ from Proposition $\ref{bscsp in proj tensor}$ $H^*$ has $w^*BSCSP.$ Hence, from (\cite{BS}, Proposition 2.1) $X \hat{\otimes}_{\pi} Y$ has $BSCSP.$
%\end{proof}
We next provide an example to show that projective tensor product of two spaces having $SD2P$ and $BDP$ respectively may not have any one of $BDP$, $BHP$ and $BSCSP.$ 
%we cannot omit the condition $BDP$ from.
\begin{example}
Let, $Z=c_0 \hat{\otimes}_{\pi} \mathbb{R}^2.$ From \cite{LLZ} it is known that every convex combination of slices of $B_Z$ has diameter two and hence $Z$ does not have $BSCSP$  while $\mathbb{R}^2$ has  $BSCSP.$ 
\end{example} 
%%%%%%%%%%%%%%%%%%%%%%%%%%%%%%%%%%%%%%%%%%%%%%%%%%%%%%%%%%%%%%

Let $X$ be a Banach space, and let $u\in S_X$ and $f\in S_{X^*}$. We put $D(X,u):=\{ h \in B_{X^*}:h(u)=1\}$ and $D^{w^*}(X^*,f):=D(X^*,f)\cap X.$ Noticing that
$D(X,u)=D^{w^*}(X^{**},u)$. We define $n(X,u)$ as the largest
non-negative real number $k$ satisfying $$k\Vert x\Vert \leq
\sup \{\vert f(x)\vert :f\in D(X,u)\}$$ for every $x\in X$. 

\begin{proposition} \label{suficiente-1} Let $X,Y$ be Banach spaces such that there exists $f\in S_{X^{*}}$ such that $n(X^{*},f)=1$ . Let  $H$ be a closed subspace of $L(X,Y)$ such
	that $X^*\otimes Y\subseteq H$ and $H^*$ has the $w^*$-$BSCSP.$ Then  $Y^*$ has the $w^*$-$BSCSP.$
\end{proposition}
\begin{proof}
	Let us assume that $H^*$ has the $w^*$-$BSCSP.$ Let $\varepsilon >0$. Then there exist $T_1,\ldots ,T_n\in S_H$, $\delta >0$ and $\lambda_1,\ldots ,\lambda_n\in
	(0,1)$ with $\sum_{i=1}^{n}\lambda_i=1$, such that $$S:=\sum_{i=1}^{n}\lambda_iS(B_{Y^*},T_i,\delta ),$$ has diameter less than $\varepsilon $. By Corollary \ref{densidad}, for each $i=1,\ldots ,n $, there exist $x_{i}\in S_X, \ y^*_{i}\in S_{Y^*}$ such that $y^*_{i}(T_i(x_{i}))>1-\delta$. We define the slices $S_{i}:=\{y^*\in B_{Y^*}:y^*(T_i(x_{i}))>1-\delta\}$ for each $i=1,\ldots ,n $. It is clear that $x_{i}\otimes S_{i}\subset S(B_{H^*},T_i,\delta )$, a hence $\sum_{i=1}^{n}\lambda_i x_i\otimes S_{i}\subset S$. As $B_X=\overline{\vert co\vert}(D^{w^*}(X^*,f))$ \cite{ABGRP} we can assume that $x_i\in D^{w^*}(X^*,f)\ \forall i\in\{1,\ldots, n\}$. If $diam (\sum_{i=1}^{n}\lambda _iS_{i})\geq \varepsilon$, then there exist $u_i,v_i\in S_i$ for $i=1,\ldots ,n $ and $y\in S_Y$ such that $y(\sum_{i=1}^{n}\lambda _iu_i-\sum_{i=1}^{n}\lambda _iv_i)\geq \varepsilon$. A hence $$(f\otimes y)(\sum_{i=1}^{n}\lambda _i x_i\otimes (u_i-v_i))\geq \varepsilon .$$ This is a contradicction, since $\sum_{i=1}^{n}\lambda _i x_i\otimes S_{i}\subset S$ and $diam (S)<\varepsilon$.
	Since $\varepsilon$ is arbitrary, we conclude that  $Y^*$ have the $w^*$-$BSCSP.$
\end{proof}

\begin{proposition} \label{suficiente-2} Let $X,Y$ be Banach spaces such that there exists $f\in S_{Y^{**}}$ such that $n(Y^{**},f)=1$. Let  $H$ be a closed subspace of $L(X,Y)$ such
	that $X^*\otimes Y\subseteq H$ and $H^*$ has the $w^*$-$BSCSP.$ Then  $X$ has the $BSCSP.$
\end{proposition}

\begin{corollary}\label{corolario-teostrong}
	Let $X,Y$ be Banach spaces. Assume that the norm of $X^*$ is
	non-rough and that there exists $f\in S_{X^{*}}$ such that
	$n(X^{*},f)=1$. Then for every closed subspace $H$ of $L(X,Y)$
	such that $X^*\otimes Y\subseteq H$, the following assertion are
	equivalent:
	\begin{enumerate}
		\item[i)] $H^*$ has the $w^*$-$BSCSP.$ \item[ii)] $Y^*$ has  the $w^*$-$BSCSP.$
	\end{enumerate}
\end{corollary}

Taking $X=\ell_1$ in the above corollary one has

\begin{corollary} Let $Y$ be Banach space. Then $L(\ell_1 ,Y)^*$ 
has the $w^*$-$BSCSP$ if and only if $Y^*$ has  the $w^*$-$BSCSP.$
\end{corollary}

Again, using the duality  $(X\widehat{\otimes}_\pi Y)^*
=L(X,Y^*)$, we get from Proposition \ref{necesaria}, a necessary
condition in order to the projective tensor product of Banach
spaces has the $BSCSP.$

\begin{corollary} Let $Y$ be Banach space. Then $\ell _1\widehat{\otimes}_\pi Y$  
	has the $BSCSP$ if and only if $Y^*$ has  the $w^*$-$BSCSP.$
\end{corollary}

%\begin{proposition}\label{necesariaBHP} Let $X,Y$ be Banach spaces with $X^{**}$ has the $w^*$-BDP and $Y^*$ has the $w^*$-BHP. Let be $H$ a closed subspace of $L(X,Y)$ such that $X^*\otimes Y\subseteq H$. Then  $H^*$ has the $w^*$-BHP.
%\end{proposition}
%\begin{proof}

%\end{proof}

%Let $X$ be a Banach space. According to \cite{HLP} we say that
%(the norm on) $X$ is {\bf{weakly octahedral}} if, for every
%finite-dimensional subspace $Y$ of $X$, every $x^*\in B_{X^*}$,
%and every $\varepsilon\in\mathbb R^+$ there exists $y\in S_X$
%satisfying

%$$\Vert x+y\Vert\geq (1-\varepsilon)(\vert x^*(x)\vert+\Vert y\Vert)\ \forall x\in Y.$$

%It is clear that octahedral norm implies weakly octahedral norm,
%but the converse is not true. In fact, by \cite[Theorem
%2.1]{BeLoRu} and \cite[Theorem 3.4]{HLP}, $(c_0\oplus_p c_0)^*$
%has weakly octahedral norm but does not have octahedral norm
%\cite[Theorem 3.2]{ABL} for every $1<p<\infty$. \vspace{0.3cm}

%{\bf PROBLEMA: X, Y tienen weakly octahedral norm. Then $X\widehat{\otimes}_\pi Y$ tiene weakly octahedral norm ?}

%Now our aim is to show an example in the setting of spaces of
%operator  enjoying weakly octahedral norm but failing to have
%octahedral norm.

%%%%%%%%%%%%%%%%%%%%%%%%%%%%%%%%%%%%%%%%%%%%%%%%%%%%%%%%

\section{Small diameter property in injective tensor products}

From symmetry of the spaces $X$ and $Y$ in the proof of the Proposition \ref{bscsp in inj tensor}, this one can be written also in the following way. 

\begin{corollary}
Let $X$ and $Y$ be Banach spaces and $H$ be a closed subspace of $L(X^*,Y)$ such that $X \otimes Y\subset H.$ Also assume that $X^*$ has $w^*$-$BSCSP$ and $Y^*$ has $w^*$-$BDP,$ then $H^*$ has $w^*$-$BSCSP.$
\end{corollary}
By considering the injective tensor product of Banach spaces $X$ and $Y$ as a subspace of the space of operators $L(X^*,Y),$ we can conclude the following.
\begin{proposition}\label{bscsp prop inj}
 Let $X$, $Y$ be Banach spaces. Suppose $X^*$ has $w^*$-$BDP$ and $Y^*$ has $w^*$-$BSCSP.$ Then $(X \hat{\otimes}_{\varepsilon} Y)^*$ has $w^*$-$BSCSP.$
 \end{proposition}
%\begin{proof}
%Follows directly from Proposition $\ref{bscsp in inj tensor},$ by considering $H= (X \hat{\otimes}_{\varepsilon} Y) \subset L(X^*,Y).$  
%\end{proof}

\begin{lemma} \label{bdp lem in inj tensor}
Let $X$ and $Y$ be Banach spaces and $H$ be a closed subspace of $L(X^*,Y).$ 
such that $X \otimes Y\subset H.$ Also assume that each $T\in H$ is $w^*-w$ continuous. If $H^*$ has $w^*$-$BDP,$ then both $X^*$ and $Y^*$ have $w^*$-$BDP.$
\end{lemma}

\begin{proof}
Let $\varepsilon>0.$ Since $H^*$ has $w^*$-$BDP,$ let $S(B_{H^*},T,\alpha)$ be a $w^*$slice in $B_{H^*}$ with diameter less than $\varepsilon$ and $\|T\|=1.$ Then there exist $x_0^*\in S_{X^*}$ and $y_0^*\in S_{Y^*}$ such that $y_0^*(Tx_0^*)>1-\frac{\alpha}{2}.$ Since $T$ is $w^*-w$ continuous, then $T^*y_0^*\in X.$ Therefore, we have $$x_0^*(T^*y_0^*)=y_0^*(Tx_0^*)>1-\frac{\alpha}{2}$$
Consider $w^*$slice $S(B_{X^*},T^*y_0^*,\frac{\alpha}{4})$ of $B_{X^*}$ and $w^*$slice $S(B_{Y^*},Tx_0^*,\frac{\alpha}{4})$ of $B_{Y^*}.$ Observe, $S(B_{X^*},T^*y_0^*,\frac{\alpha}{4})\otimes \{y_0^*\}\subset S(B_{H^*},T,\alpha).$ Indeed,  let $(x^*,y_0^*)\in S(B_{X^*},T^*y_0^*,\frac{\alpha}{4})\otimes \{y_0^*\}.$ Then 
$$  T(x^*,y_0^*)=y_0^*(Tx^*)=x^*(T^*y_0^*)>\|T^*y_0^*\|-\frac{\alpha}{4}=1-(1-\|T^*y_0^*\|+\frac{\alpha}{4})>1-\alpha $$
Similarly, $\{x_0^*\} \otimes S(B_{Y^*},Tx_0^*,\frac{\alpha}{4})\subset S(B_{H^*},T,\alpha).$ \\
Hence, $diam (S(B_{X^*},T^*y_0^*,\frac{\alpha}{4}))<\varepsilon$ and $diam(S(B_{Y^*},Tx_0^*,\frac{\alpha}{4}))<\varepsilon.$ Thus, both $X^*$ and $Y^*$ have $w^*$-$BDP.$
\end{proof}

%Since $H=X\hat{\otimes}_{\varepsilon} Y$ is a subspace of $L(X^*,Y)$ with each element $T\in X\hat{\otimes}_{\varepsilon} Y$ is $w^*-w$ continuous, then we can conclude the following :
In view of Propositions $\ref{rough}$, $\ref{bscsp in inj tensor}$ and Lemma $\ref{bdp lem in inj tensor}$ we get the following.
\begin{proposition} \label{bdp in inj tensor}
Let $X$ and $Y$ be Banach spaces. Then norm of $X\hat{\otimes}_{\varepsilon} Y$ is non-rough if and only if norm of both $X$ and $Y$ are non-rough.
\end{proposition}

%\begin{example}
%Let, $Z=l_1 \hat{\otimes}_{\varepsilon} \mathbb{R}^2.$ From (\cite{LLZ}, Corollary 3.3) it is known that every convex combination of $w^*$ slices of $B_{Z^*}$ has diameter two and hence $Z^*$ does not have $w^*BSCSP$ while $\mathbb{R}^2$ has both $w^*BHP$ and $w^*BSCSP.$ 
%\end{example}  
However, the analogous result for projective tensor product is not true in general.
%However, similar results for projective tensor product are not true. Explicitly, in following example we will show that Propositions $\ref{bscsp prop proj},$ $\ref{bdp prop proj}$ and $\ref{bhp prop proj}$ fails in the case of $w^*BDP$, $W^*BHP$ and $w^*BSCSP.$
%Although $BDP$, $BHP$ and $BSCSP$ are preserved by projective tensor product from  one of the factor, provided the other factor has $BDP,$ but this is not true when we consider the case $w^*BDP$, $W^*BHP$ and $w^*BSCSP.$
\begin{example}  If the norm of $l_2\hat{\otimes}_{\pi} l_2$ is non-rough, then $L(\ell_2)$  has slices of arbitrarily small diameter, but it is well known that all slices of $L(\ell_2)$ have diameter two \cite{BeLoPeRo}.
%		From \cite[Corollary III.1.3]{HWW} (Corollary III.1.3) and Principle of local reflexivity 
Therefore it follows that $l_2\hat{\otimes}_{\pi} l_2$ does not have non-rough norm, but $l_2$ is a Hilbert space and hence it has non-rough norm.
\end{example}

\begin{Acknowledgement}
The third  author's research is funded by the National Board for Higher Mathematics (NBHM), Department of Atomic Energy (DAE), Government of India, Ref No: 0203/11/2019-R$\&$D-II/9249.
\end{Acknowledgement}

\begin{Conflict of interest statement and data availability}
Not Applicable.
\end{Conflict of interest statement and data availability}

\end{document}